\documentclass[12pt, a4paper]{paper}
\usepackage[utf8]{inputenc}
\usepackage{setspace}
\usepackage{amssymb,amsthm,latexsym}
\usepackage{nicefrac}
\usepackage{amsmath}
\usepackage{array}
\usepackage[displaymath,mathlines, running]{lineno}
\usepackage[margin=2.5cm]{geometry}
\usepackage{blindtext}
\usepackage{epstopdf}
\usepackage{amsfonts}
\usepackage{parskip}
\usepackage[numbers]{natbib}
\usepackage{epstopdf}
\usepackage{multirow}
\usepackage{ragged2e}
\usepackage{hyperref}
\usepackage[section]{placeins}
\usepackage[labelformat=simple]{subcaption}
\usepackage{lipsum}
\usepackage[english]{babel}
\usepackage{enumitem}
\usepackage{float}
\usepackage{authblk}
\usepackage{todonotes}
\usepackage{orcidlink}
\usepackage{fontawesome5}

\newtheorem{theorem}{Theorem}[section]

\newtheorem{definition}[theorem]{Definition}
\newtheorem{proposition}[theorem]{Proposition}
\newtheorem{lemma}[theorem]{Lemma}

\newtheorem{corollary}[theorem]{Corollary}
\newtheorem{observation}[theorem]{Observation}

\newtheorem{conjecture}[theorem]{Conjecture}

\renewenvironment{proof}{\noindent\textbf{Proof:}}{\qed}

\def\mp{\operatorname{mp}}
\def\disc{\operatorname{disc}}

\title{Multipacking in Hypercubes\footnote{All the authors contributed equally.}}
\small
\author[1]{Deepak \textbf{Rajendraprasad} \href{mailto:deepak@iitpkd.ac.in}{\faEnvelope} \orcidlink{0000-0001-9101-8967}}

\author[1]{Varun \textbf{Sani} \href{mailto:varunsani625@gmail.com}{\faEnvelope}} 

\author[1]{Birenjith \textbf{Sasidharan} \href{mailto:biren@iitpkd.ac.in}{\faEnvelope} \orcidlink{0000-0001-7444-7161}}

\author[1]{Jishnu \textbf{Sen} \href{mailto:senjishnu5@gmail.com}{\faEnvelope} \orcidlink{0000-0002-8724-9583}}

\affil[1]{Indian Institute of Technology Palakkad, Kerala, India, 678623}

\date{\today}

\begin{document}
\maketitle

\newcommand*{\prf}{\textbf{Proof of the lower bound of Theorem \ref{r5}:}}
\newcommand*{\prff}{\textbf{Proof of the upper bound of Theorem \ref{r5}:}}
\newcommand*{\prof}{\textbf{Proof of Lemma \ref{r10}:}}

\begin{abstract}
\noindent For an undirected graph $G$, a \emph{dominating broadcast} on $G$ is a function $f : V(G) \rightarrow \mathbb{N}$ such that for any vertex $u \in V(G)$, there exists a vertex $v \in V(G)$ with $f(v) \geqslant 1$ and $d(u,v) \leqslant f(v)$. The \emph{cost} of $f$ is $\sum_{v \in V} f(v)$. The minimum cost over all the dominating broadcasts on $G$ is defined as the \emph{broadcast domination number} $\gamma_b(G)$ of $G$. A \emph{multipacking} in $G$ is a subset $M \subseteq V(G)$ such that, for every vertex $v \in V(G)$ and every positive integer $r$, the number of vertices in $M$ within distance $r$ of $v$ is at most $r$. The \emph{multipacking number} of $G$, denoted $\mp(G)$, is the maximum cardinality of a multipacking in $G$. These two optimisation problems are duals of each other, and it easily follows that $\mp(G) \leqslant \gamma_b(G)$. It is known that $\gamma_b(G) \leqslant 2\mp(G)+3$ and conjectured that $\gamma_b(G) \leqslant 2\mp(G)$. 

In this paper, we show that for the $n$-dimensional hypercube $Q_n$
$$
\left\lfloor\frac{n}{2} \right\rfloor 
    \leqslant \mp(Q_n) 
    \leqslant \frac{n}{2} + 6\sqrt{2n}.
$$

Since $\gamma_b(Q_n) = n-1$ for all $n \geqslant 3$, this verifies the above conjecture on hypercubes and, more interestingly, gives a sequence of connected graphs for which the ratio $\frac{\gamma_b(G)}{\mp(G)}$ approaches $2$, a search for which was initiated by Beaudou, Brewster and Foucaud in 2018. It follows that, for connected graphs $G$
$$
    \limsup_{\mp(G) \rightarrow \infty} \left\{\frac{\gamma_b(G)}{\mp(G)}\right\} = 2.$$ 

The lower bound on $\mp(Q_n)$ is established by a recursive construction, and the upper bound is established using a classic result from discrepancy theory.
\end{abstract}

\hrule
\vspace{0.5cm}

\section{Introduction}

A \emph{dominating set} in an undirected graph $G$ is a set $S \subseteq V(G)$ such that every node in $G$ is either in $S$ or adjacent to a node in $S$. One of the many motivations to study dominating sets and its variants comes from optimising the placement of facilities on the nodes of a network so that services can be easily distributed to every node in the network. In particular, if placing a facility at a node serves that node and all its neighbours, and the cost of establishing a facility is same across all the nodes, then the best strategy is to identify a smallest dominating set of the graph and place one facility on each node of this set. Suppose we can set up facilities that can serve a larger range, even at a larger cost, then we might be able to do a better distribution than the above strategy. In particular, if the cost of setting up a facility is proportional to the distance up to which it can serve, then the task becomes that of finding a dominating broadcast of minimum cost as defined next. 
\begin{definition}
    A \emph{broadcast} on a graph $G = (V, E)$ is a function $f: V \to \mathbb{N}$. The \emph{cost} of a broadcast $f$ is $\sum_{v \in V} f(v)$. A broadcast is said to be \emph{dominating} if every vertex of $G$ lies within a distance $f(v)$ of some vertex $v \in V$ with $f(v) \geqslant 1$. The \emph{broadcast domination number}, denoted $\gamma_b(G)$, is the minimum cost over all dominating broadcasts on $G$. 
\end{definition}



Trees of radius $2$ give an example of a family of graphs where $\gamma(G)$ (the size of a smallest dominating set in $G$) is unbounded but $\gamma_b(G)$ is at most $2$. The notion of broadcast domination was introduced by \citet{erwinintro} in 2001 under the name of \textit{cost domination}. It is easy to see that the broadcast domination number $\gamma_b(G)$ is bounded above by both the radius of $G$ and $\gamma(G)$. Erwin showed that $\gamma_b(G)$ is bounded below by $(\operatorname{diam}(G)+1)/3$. \citet{heggernes2006optimal} in 2006 showed, quite contrary to the usual case for domination problems, that the broadcast domination number of a graph can be determined in polynomial time. In 2013, \citet{brewster2013new} modelled broadcast domination as an Integer Linear Program (ILP), relaxed it to a Linear Program (LP), and used the ILP strengthening of the dual LP to find lower bounds on broadcast domination number for some graphs. This ILP strengthening gave the cute combinatorial problem of multipacking that we define next. Here $N_k[v]$ denotes the set of vertices which are at a distance at most $k$ from $v$. 


\begin{definition}
    For a graph $G = (V,E)$, a set $M \subseteq V$ is a  \emph{multipacking} in $G$ if for every vertex $v \in V$, $|N_k[v] \cap M| \leqslant k$ for all $k \geqslant 1$. The \emph{multipacking number} $\mp(G)$ is the maximum cardinality of a multipacking in $G$.
\end{definition}

From the LP duality, we get that $\mp(G) \leqslant \gamma_{b}(G)$ for every graph $G$. On the other hand, \citet{hartnell2014} in 2014 proved that $\gamma_{b}(G) \leqslant 3\mp(G)-2$ for any graph $G$ with $\mp(G) \geqslant 2$. \citet{beaudou2018broadcast} in 2019 improved this to $\gamma_{b}(G) \leqslant 2\mp(G) + 3$ and conjectured that the additive factor of $3$ can be removed from this bound.

\begin{conjecture}\emph{\cite{beaudou2018broadcast}} \label{r4}
    For any graph $G$, $\gamma_{b}(G) \leqslant 2\mp(G)$.
\end{conjecture}

The reason for the multiplier of $2$ in the above conjecture is that there are a few small graphs (including $C_4$, $C_5$) where $\mp(G) = 1$ and $\gamma_b(G) = 2$ and a few others where $\mp(G) = 2$ and $\gamma_b(G) = 4$. By taking disjoint copies of these examples, one can construct, for any $k \geqslant 1$, a graph with $\mp(G) = k$ and $\gamma_b(G) = 2k$. Hence, the above conjecture, if true, is tight. While noting this, \citet{beaudou2018broadcast} lamented that we do not have an infinite family of connected graphs where the ratio $\frac{\gamma_b(G)}{\mp(G)}$ approaches $2$. The best construction known so far is by \citet{hartnell2014} who constructed an infinite family of connected graphs where $\frac{\gamma_b(G)}{\mp(G)} = \frac{4}{3}$. 

\subsection{Results}

The main contribution of this note is to show that hypercubes form an infinite family of connected graphs where $\frac{\gamma_b(G)}{\mp(G)}$ approaches $2$. An  $n$-dimensional hypercube $Q_n$ is the Cartesian product of $n$ copies of the complete graph $K_{2}$. Alternatively, it can be visualized as a graph whose vertex set is $\{0,1\}^{n}$ and two vertices are adjacent if exactly one coordinate is different. Our main result is

\begin{theorem}\label{r5}
    For any positive integer $n$, 
    $$
    \left\lfloor\frac{n}{2}\right\rfloor 
        \leqslant \mp(Q_{n}) 
        \leqslant \frac{n}{2} + 6\sqrt{2n}.
    $$
\end{theorem}

\citet{brevsarproductgraph} showed in 2019 that $\gamma_b(Q_n) = n-1$ for $n \geqslant 3$ and $\gamma_b(Q_n) = n$ for $n \in \{1,2\}$. Since $n - 1 \leqslant 2\left\lfloor \frac{n}{2} \right\rfloor$, the lower bound in Theorem~\ref{r5} proves Conjecture~\ref{r4} on hypercubes. More interestingly, we see that $\lim\limits_{n \rightarrow \infty} \frac{\gamma_b(Q_n)}{\mp(Q_n)} = 2$. This, together with the upper bound $\gamma_b(G) \leqslant 2\mp(G) + 3$ \cite{beaudou2018broadcast} lets us conclude

\begin{corollary}\label{corLimSup}
    For all connected graphs $G$,
    $$
    \limsup_{\mp(G) \rightarrow \infty} \left\{
        \frac{\gamma_b(G)}{\mp(G)}\right\} = 2.
    $$
        
\end{corollary}

\subsection{Proof Techniques}

The lower bound in Theorem \ref{r5} is proved by introducing a recursive technique that systematically generates multipacking of higher-dimensional hypercubes by combining multipackings from lower dimensions. For the proof of the upper bound, we bank on a classic result by   \citet{spencer1985six} from combinatorial discrepancy theory.

\subsection{Related Results}

The first inequality in $\mp(G) \leqslant \gamma_b(G) \leqslant 2\mp(G) + 3$ was shown to be tight in some graph families like trees \cite{teshima2012broadcasts}, grid graphs $P_{m} \square P_{n}$ (except $(m,n) \neq (4,6)$) \cite{brewster2019multipacking} and strongly chordal graphs \cite{brewster2019broadcast}. A graph is \textit{strongly chordal} if it is chordal and every even cycle of length at least $6$ has a chord that connects two vertices which are at an odd distance apart on the cycle. 

\citet{hartnell2014} shown that the difference between $\mp(G)$ and $\gamma_b(G)$ can be arbitrarily large by constructing an infinite family of connected graphs $G$ such that $\frac{\gamma_b(G)}{\mp(G)}=\frac{4}{3}$. This construction and the upper bound $\gamma_b(G) \leqslant 2mp(G) + 3$ meant that for connected graphs $G$
$$
    \frac{4}{3} \leqslant 
        \limsup_{\mp(G) \rightarrow \infty} \left\{\frac{\gamma_b(G)}{\mp(G)}\right\} 
        \leqslant 2.
$$ 
While our Corollary~\ref{corLimSup} improves the above, optimal bounds for this ratio were studied for special graph classes.  For connected chordal graphs $G$, \citet{das2023relation} showed that 
$$
    \frac{10}{9} 
    \leqslant \limsup_{\mp(G) \rightarrow \infty} 
        \left\{\frac{\gamma_b(G)}{\mp(G)}\right\} 
    \leqslant \frac{3}{2}.
$$ 
A \textit{cactus} is a connected graph in which any two cycles share at most one vertex. A graph $G$ is a \textit{$\delta$-hyperbolic graph}, if for any four vertices $u, v, w, x$ of $G$, among the three sumations $d(u,v)+d(w,x), d(u,w)+d(v,x)$ and $d(u,x)+ d(v,w)$, the difference between the two of the largest sums is at most $2\delta$. A graph class is said to be \textit{hyperbolic} if there exists a constant $\delta$ such that every graph in that class is $\delta$-hyperbolic. \citet{das2025multipacking} proved that for cactus graphs and $\frac{1}{2}$-hyperbolic graphs $G$,  
$$
    \frac{4}{3} 
        \leqslant \limsup_{\mp(G) \rightarrow \infty}
            \left\{\frac{\gamma_b(G)}{\mp(G)}\right\} 
        \leqslant \frac{3}{2}.
$$


\subsection{Terminology}

Every graph discussed in this note is finite, simple, and undirected. $\mathbb{N}$ denotes the set of natural numbers (including $0$). For any positive integer $n$, we denote the set $\{1,2, \dots, n\}$ as $[n]$. Any undefined terms and notations are in accordance with \citet{chartrand2010graphs}.

\section{Proof of Theorem \ref{r5}} \label{s2}
We use exponential notation to indicate repeated sequences in the vertex so that a vertex $110001$ in $Q_6$ will be written as $1^2 0^3 1$. The \textit{hamming weight} $\operatorname{wt}(u)$ of a vertex $u$ is the distance of $u$ from $0^{n}$ in $Q_n$. As a warm-up, first we determine $\mp(Q_{n})$ for $n \leqslant 6$. It is easy to observe that  $\mp(Q_{1})=1, \mp(Q_{2})=1, \mp(Q_{3})=2, \mp(Q_{4})=2$ and $\mp(Q_{5})=2$. 
\begin{proposition}\label{propQ6}
$\mp(Q_{6})=4$.
\end{proposition}
\begin{proof}
One can observe that, since the set $\{0^6, 0^31^3, 1^30^3, 1^6\}$ is a multipacking in $Q_{6}$, $\mp(Q_{6}) \geqslant 4$. Now, we show that no set of order $5$ can be a multipacking in $Q_{6}$. On the contrary, let $P$ be a multipacking in $Q_{n}$ of $5$ vertices. As $Q_{n}$ is vertex-transitive, without loss of generality, suppose $0^{6} \in P$. Then $P$ cannot contain any vertex of hamming weight $1$ or $2$. Then $P \setminus \{0^{6}\} \subseteq N_{3}[1^{6}]$. Hence, $|N_{3}[1^{6}] \cap P|=4 > 3$, a contradiction to the fact that $P$ is a multipacking. Therefore, $\mp(Q_{6})=4$. 
\end{proof}

Since $Q_{n+1}$ contains a copy of $Q_n$ as a distance preserving subgraph, it is easy to observe that the multipacking number of $Q_n$ is monotonic in $n$. 
\begin{observation}\label{r0}
For any positive integer $n$, $\mp(Q_{n}) \leqslant \mp(Q_{n+1})$.
\end{observation}

\subsection{Lower bound}
We begin with a couple of lemmas needed for the recursive construction. Given a vertex \( x \in Q_n \), and an integer \( m \geqslant 1 \), we define \( i^m \cdot x \) as the binary string obtained by concatenating $i^m$ to \( x \), where $i \in \{0,1\}$. For any set of vertices $S$, $$i^m \cdot S=\{i^m \cdot x: x \in S\}.$$


\begin{lemma}\label{r10}
Let $n_{0} \geqslant n_{1}$ be two positive integers, and $Q_{n_0}$ and $Q_{n_1}$ be two hypercubes equipped with multipackings $P_{0}$ and $P_{1}$, respectively. Further, let 
$$
    n = n_{0} + \max(|P_0|,2) + \max(|P_1|,2) -1,
$$ 
and $P$ be the set $$
(0^{n - n_{0}} \cdot P_{0}) \cup (1^{n - n_{1}} \cdot P_{1}).
$$ Then $P$ is a multipacking in $Q_n$.
\end{lemma}

\begin{proof}
Let $p = |P| = |P_{0}|+|P_{1}|$. Pick any $x \in V(Q_n)$ and any $k \in [p-1]$. We will show that the number of vertices of $P$ in $N_{k}[x]$ is at most $k$ by counting separately the number of vertices of $0^{n - n_{0}} \cdot P_{0}$ and $1^{n - n_{1}} \cdot P_{1}$ in $N_{k}[x]$.

Let $x_0$ denote the $n_0$-length suffix of $x$ (the last $n_0$ bits) and $x_1$ denote the $n_1$-length suffix of $x$. Let $q_0$ and $q_1$ respectively denote the number of zeros and ones in the first $(n - n_0)$ bits of $x$. For any vertex $y \in P_{0}$,  $d_{Q_n}(x,0^{n - n_{0}} \cdot y) = q_1 + d_{Q_{n_0}}(x_0, y)$. For any vertex $y \in P_{1}$, $d_{Q_n}(x,1^{n - n_{1}} \cdot y) \geqslant q_0 + d_{Q_{n_1}}(x_1, y)$. Hence we have
\begin{align} 
    |N_k[x] \cap P|
        & = |N_k[x] \cap 0^{n-n_{0}} \cdot P_{0}| 
            + |N_k[x] \cap 1^{n-n_{1}} \cdot P_{1}| \notag \\
        &\leqslant |N_{k-q_1}[x_0] \cap P_{0}|
            + |N_{k-q_0}[x_1] \cap P_{1}|. \label{eq2}
\end{align}

If both $(k - q_1)$ and $(k-q_0)$ are positive, then the right hand side of (\ref{eq2}) is bounded above by $(k - q_1) + (k - q_0)$, since $P_i$ is a multipacking in $Q_{n_i}$ for each $i \in \{0,1\}$. Since $(q_0 + q_1) = (n - n_0) \geqslant (p - 1) \geqslant k$, this bound is at most $k$ and we are done. If $(k - q_1) < 0$, then the right hand side of (\ref{eq2}) is bounded above by $0 + (k - q_0) \leqslant k$. The case when $(k - q_0) < 0$ is similar.
If $(k - q_1) = 0$ but $q_0 > 0$, then the right hand side of (\ref{eq2}) is bounded above by $1 + (k - q_0) \leqslant k$. The case when $(k - q_0) = 0$ but $q_1 > 0$ is also similar.

We are only left with two boundary cases, viz., $k - q_1 = 0 = q_0$ and $k - q_0 = 0 = q_1$. Here we use the fact that $n - n_0 = \max(|P_0|,2) + \max(|P_1|,2) -1 \geqslant \max\{|P_0|, |P_1|\} + 1$. In the first boundary case, since $q_0 = 0$, we have $q_1 = (n - n_0) \geqslant |P_1| + 1$. Further since $k - q_1 = 0$ in this case, we have $k = q_1 \geqslant |P_1| + 1$. Hence, the right-hand side of (\ref{eq2}), which is at most $1 + |P_1|$, is bounded above by $k$. The second boundary case ($k - q_0 = 0 = q_1$) is similar.
\end{proof}


\begin{corollary}\label{r13}
Let $n_{0} \geqslant n_{1}$ be two positive integers. Let $\mp(Q_{n_0}) \geqslant p_0$ and $\mp(Q_{n_1}) \geqslant p_1$. Then $\mp(Q_n) \geqslant p_0+p_1$, where 
$$n = n_{0} + \max(p_0,2) + \max(p_1,2) - 1.$$
\end{corollary}
This recursive approach leads to a general lower bound of the multipacking number on hypercubes, which we formalize in the subsequent results. 
\begin{lemma}\label{r11}
For any positive integer $k$, $\mp(Q_{2k}) \geqslant k$.
\end{lemma}
\begin{proof}
We use induction on $k$. The statement is easy to verify for $k \leqslant 2$ and holds for $k=3$ by Proposition~\ref{propQ6}. Suppose $k \geqslant 4$ and that the statement holds for all natural numbers less than $k$. Due to the induction hypothesis, we have
$$\mp\left(Q_{2\left\lceil\frac{k}{2}\right\rceil}\right) \geqslant \left\lceil\frac{k}{2}\right\rceil,~ \text{and} \mp\left(Q_{2\left\lfloor\frac{k}{2}\right\rfloor}\right) \geqslant \left\lfloor\frac{k}{2}\right\rfloor.$$

Since $\left\lfloor\frac{k}{2}\right\rfloor \geqslant 2$, by Corollary \ref{r13}, we have 
$$
    \mp(Q_n) 
        \geqslant  \left\lceil\frac{k}{2}\right\rceil
        +           \left\lfloor\frac{k}{2}\right\rfloor\\
        = k,
$$
where
$$
    n   =  2 \left\lceil\frac{k}{2}\right\rceil
        +   \left\lceil\frac{k}{2}\right\rceil
        +   \left\lfloor\frac{k}{2}\right\rfloor
        -   1 
        \leqslant 2k.
$$
Hence by Observation \ref{r0}, we have $\mp(Q_{2k}) \geqslant k$.
\end{proof}

\noindent\prf ~If $n=1$, then $\mp(Q_1) > 0$. When $n$ is an even positive integer, the proof follows from Lemma \ref{r11}. When $n=2k+1$ is odd for some positive integer $k$, by Lemma \ref{r11} and Observation \ref{r0} we have, 
$$\mp(Q_n) = \mp(Q_{2k+1}) \geqslant \mp(Q_{2k}) \geqslant k = \left\lfloor \frac{n}{2} \right\rfloor.$$ \qed

Though we cannot improve the lower bound of $\left\lfloor \frac{n}{2} \right\rfloor$ in general, we show that $\mp(Q_n) - \left\lfloor \frac{n}{2} \right\rfloor$ can be arbitrarily large.
\begin{proposition}
For every positive integer $i$, $\mp(Q_{n_i}) \geqslant \frac{n_i}{2} + \frac{\log_2 n_i - 1}{2}$, where $n_i = 2^{i+1} - i$.
\end{proposition}
\begin{proof}
We construct a specific sequence of hypercubes \( \{Q_{n_i}\} \) by repeatedly applying Corollary \ref{r13} starting from \( Q_3 \), for which the multipacking number is $2$. Hence, we consider \( n_1 = 3 \) and \( p_1 = 2 \). At each step, we consider two identical copies multipacking of the hypercube $Q_{n_{i}}$, and using Lemma $\ref{r10}$, we obtain a multipacking of $Q_{n_{i+1}}$. We continue this process indefinitely and obtain the recurrence relations
\[
n_i = n_{i-1} + 2p_{i-1} - 1, \quad p_i = 2p_{i-1}.  
\] 
On solving these recurrence relations with initial conditions \( n_1 = 3 \) and \( p_1 = 2 \), we obtain
\[
n_i = 2^{i+1} - i , \quad p_i = 2^i \quad\text{for all } i \geqslant 1.
\] From $n_i = 2^{i+1} - i$, we get $p_i = \frac{n_i}{2} + \frac{i}{2}$, and using $n_i \leqslant 2^{i+1}$, it follows that $i \geqslant \log_2 n_i - 1$. As \( \mp(Q_{n_i}) \geqslant p_i  \), we have $$\mp(Q_{n_i}) \geqslant \frac{n_i}{2} + \frac{\log_2 n_i - 1}{2}.$$
\end{proof}

\subsection{Upper bound}
Suppose we have a finite family of sets with finite elements, and we intend to color the underlying set with two colors such that each subset has roughly half of each color. The discrepancy quantifies how unbalanced any set in the family can be under the best possible two-coloring of the underlying set. Formally, let $\mathcal{A}$ be a family of subsets of $\Omega$, and consider the coloring as a mapping $$\chi: \Omega \rightarrow \{-1,+1\}.$$ Suppose for every $A \subseteq \Omega$, $\chi(A)=\sum_{a \in A} \chi(a)$. Then, the discrepancy of $\mathcal{A}$ with respect to $\chi$ is defined as $$\disc(\mathcal{A}, \chi) = \max_{A \in \mathcal{A}} |\chi(A)|.$$ The discrepancy of $\mathcal{A}$ is defined as $$\disc(\mathcal{A}) = \min_{\chi: \Omega \rightarrow\{-1,+1\}} \disc(\mathcal{A}, \chi).$$ \\[0.5em]
Let's recall a classic result due to \citet{spencer1985six} from the discrepancy theory. 
\begin{theorem}\emph{\cite{spencer1985six}} \label{disc}
	Let $\mathcal{A}$ be a family of $n$ subsets of an $n$-element set $\Omega$. Then $$\disc(\mathcal{A}) \leqslant 6\sqrt{n}.$$
\end{theorem}

The next lemma establishes an upper bound on $\mp(Q_n)$ using this bound. 

\begin{lemma}\label{lemUB}
For any positive integer $n$, we have
$$\mp(Q_n) 
    \leqslant \frac{n}{2} + 6\sqrt{2 \mp(Q_n)}
$$
\end{lemma}

\begin{proof}
Let $P=\{x_1, x_2, \dots, x_p\}$ be a maximum multipacking in $Q_n$ of size $p ~(\geqslant \frac{n}{2})$. For each $x_i \in P$, we construct a set $A_i \subseteq [n]$ comprising of the coordinates at which $x_i$ has $1$, and $\overline{A_i}= [n] \setminus A_i$. In particular, $\overline{A_i}$ contains the coordinate values at which $x_i$ has $0$. Now, consider the family of $2p$ sets $$\mathcal{A}=\{A_1, \overline{A_1}, A_2, \overline{A_2}, \dots, A_{p}, \overline{A_{p}}\},$$ with the underlying set $\Omega = [2p] \supseteq [n]$. Let $\chi$ be a mapping such that $\disc(\mathcal{A}, \chi) = \disc(\mathcal{A})$. For each $i \in [n]$, if $\chi(i) = -1$, then we flip the bit at the $i$-th coordinate for all the vertices of $Q_n$. All such flippings induce an automorphism $\varphi$ on the vertex set of $Q_n$. Hence, the set $P$ contains new vertices of $Q_n$. Further, a bit-flipping preserves the hamming distance, and therefore $P$ is still a multipacking in $Q_n$.
    
    Since, for each $i \in [n]$, $|\chi(A_i)|, |\chi(\overline{A_i})| \leqslant \disc(\mathcal{A})$, the number of ones in  $\varphi(x_i)$ is 
    at most $\left(\frac{|A_i|}{2} + \frac{\disc(\mathcal{A})}{2}\right) + \left(\frac{|\overline{A_i}|}{2} + \frac{\disc(\mathcal{A})}{2}\right)$ = $\frac{n}{2} +
	\disc(\mathcal{A})$. Hence $\operatorname{wt}(x_i) \leqslant \frac{n}{2} + \disc(\mathcal{A})$ for every $x_i \in P$ after flipping. 
	This means that, after flipping, $P \subseteq N_{\frac{n}{2} + \disc(\mathcal{A})}[0^n]$. As $P$ is still a multipacking, we have $$|P| \leqslant \frac{n}{2} + \disc(\mathcal{A}) \leqslant \frac{n}{2} + 6\sqrt{2p},$$ where the last inequality is due to Theorem \ref{disc}. As $P$ is a maximum multipacking in $Q_n$, we have the desired upper bound. 
\end{proof}

\noindent \prff~ The upper bound $\mp(Q_n) \leqslant \frac{n}{2} + 6\sqrt{2n}$ in Theorem~\ref{r5} now follows from Lemma~\ref{lemUB} since  $\mp(Q_n) \leqslant n$. \qed

\bibliographystyle{abbrvnat}
\bibliography{paper}

\end{document}